\documentclass[12pt]{amsart}

\usepackage{amsmath,amssymb,enumerate,mathtools,graphicx}
\newtheorem{theorem}{Theorem}[section]
\newtheorem{claim}{Claim}[theorem]
\newtheorem{lemma}[theorem]{Lemma}

\newtheorem{corollary}[theorem]{Corollary}

\theoremstyle{definition}

\newcommand{\bF}{\mathbb F}
\newcommand{\bR}{\mathbb R}
\newcommand{\bZ}{\mathbb Z}

\newcommand{\cG}{\mathcal{G}}
\newcommand{\cL}{\mathcal{L}}

\newcommand{\cM}{\mathcal{M}}

\newcommand{\cQ}{\mathcal{Q}}

\newcommand{\cT}{\mathcal{T}}
\newcommand{\cU}{\mathcal{U}}

\newcommand{\smalltri}{\raisebox{0.05ex}{\scalebox{0.4}{$\triangle$}}}
\newcommand{\smallsq}{\scalebox{0.4}{$\square$}}

\newcommand{\freeexttech}[1]{M_{#1}^{\circ}}
\newcommand{\triangleexttech}[1]{M_{#1}^{\smalltri}}
\newcommand{\fanoexttech}[1]{M_{#1}^{\smallsq}}
\newcommand{\freeext}[1]{\freeexttech{#1}}
\newcommand{\fanoext}[1]{\fanoexttech{#1}}
\newcommand{\triangleext}[1]{\triangleexttech{#1}}

\newcommand{\mcsigned}{\cG^{\smalltri}}
\newcommand{\mcevencycle}{\cG^{\smallsq}}
\newcommand{\mcfree}{\cG^{\circ}}
\DeclareMathOperator{\si}{si}

\DeclareMathOperator{\cl}{cl}

\DeclareMathOperator{\PG}{PG}

\DeclareMathOperator{\GF}{GF}
\DeclareMathOperator{\AG}{AG}
\newcommand{\elem}{\varepsilon}
\newcommand{\del}{\!\setminus\!}
\newcommand{\con}{/}
\newcommand{\twothirds}[1]{\left\lceil {2 {#1}}/{3} \right\rceil}
\begin{document}
\sloppy
\author[Geelen]{Jim Geelen}
\author[Nelson]{Peter Nelson}
\address{Department of Combinatorics and Optimization,
University of Waterloo, Waterloo, Canada}
% \email{jfgeelen@math.uwaterloo.ca}
\thanks{ This research was partially supported by a grant from the
Office of Naval Research [N00014-12-1-0031].}
\title[Quadratic growth-rates]{Matroids denser than a clique}
\begin{abstract}
The {\em growth-rate function} for a minor-closed class $\cM$ of matroids is the
function $h$ where, for each non-negative integer $r$,
$h(r)$ is the maximum number of elements of a simple matroid in $\cM$ 
with rank at most $r$. The Growth-Rate Theorem of Geelen, Kabell, Kung, and Whittle
shows, essentially, that  the growth-rate function is always either linear, quadratic, exponential, or infinite.
Moreover, if the growth-rate function is quadratic, then $h(r)\ge \binom{r+1}{2}$, with the
lower bound coming from the fact that such classes necessarily contain all 
graphic matroids. We characterise the classes that satisfy $h(r) = \binom{r+1}{2}$ 
for all sufficiently large $r$. 
\end{abstract}
\date{\today}
\maketitle
\section{Introduction}
A \emph{single-element extension} of a matroid $M$ by an element $e \notin E(M)$ is a matroid $M'$ such that $M = M' \del e$. A single-element extension of $M \cong M(K_{n+1})$ by $e$ is nongraphic if and only if $e$ is not a loop or a coloop or parallel to any other element of $M$. We prove the following theorem: 
\begin{theorem}\label{main}
Let $n \ge 2$ and $\ell \ge 3$ be integers. If $M$ is a simple matroid of sufficiently large rank with $|M| > \binom{r(M)+1}{2}$, then $M$ has a minor isomorphic to either $U_{2,\ell+2}$ or a nongraphic single-element extension of $M(K_{n+1})$. 
\end{theorem}

This theorem is closely related to the problem of determining growth-rates of minor-closed classes. For a class $\cM$ of matroids containing the empty matroid, let $h_{\cM}(n): \bZ_0^+ \to \bZ_0^+ \cup \{\infty\}$ denote the \emph{growth-rate function} of $\cM$: the function whose value at an integer $n \ge 0$ is given by the maximum number of elements in a simple matroid in $\cM$ of rank at most $n$. For example, the class $\cG$ of graphic matroids has growth-rate function $h_{\cG}(n) = \binom{n+1}{2}$. Any class containing all simple rank-$2$ matroids has infinite growth-rate function for all $n \ge 2$; the following theorem of Geelen, Kabell, Kung and Whittle (see [\ref{gkw09}]) determines all growth-rate functions to within a constant factor. 
To simplify the statement of this and other results, we will take the convention that 
{\em minor-closed} classes of matroids are closed under both minors and isomorphism.

\begin{theorem}[Growth-rate Theorem]
If $\cM$ is a nonempty minor-closed class of matroids not containing all simple rank-$2$ matroids, then there exists $c \in \bR$ so that either:
\begin{enumerate}
\item $h_{\cM}(n) \le c n$ for all $n$, 
\item\label{qdense} $\binom{n+1}{2} \le h_{\cM}(n) \le cn^2$ for all $n$ and $\cM$ contains all graphic matroids, or 
\item there is a prime power $q$ such that $\frac{q^n-1}{q-1} \le h_{\cM}(n) \le cq^n$ for all $n$ and $\cM$ contains all $\GF(q)$-representable matroids.
\end{enumerate}
\end{theorem}

Minor-closed classes satisfying (\ref{qdense}) are \emph{quadratically dense}. If $f$ and $g$ are functions, then we write $f(n) \approx g(n)$ if $f(n) = g(n)$ for all but finitely many $n$. Theorem~\ref{main} will imply a stronger result, Theorem~\ref{strongermain}, which in turn implies the following theorem, giving a `gap' in which no growth-rate function can fall.    

\begin{theorem}\label{maingap}
Let $\cM$ be a quadratically dense minor-closed class of matroids. Either $h_{\cM}(n) \approx \binom{n+1}{2}$, or $h_{\cM}(n) \ge \binom{n+2}{2}-3$ for all $n \ge 2$. 
\end{theorem}

Similar behaviour has been shown to occur in the `exponentially dense' case; see Nelson~[\ref{nthesis}, Theorem 1.5.13].

Exponentially dense classes are easier to work with than quadratically dense classes since 
the extremal matroids are very highly connected; see [\ref{nthesis}, Theorem 1.5.6]. In fact, this connectivity is a very strong variety that is not lost by contraction. For quadratically dense classes, one can show that 
the extremal matroids are highly connected and that  there are ``useful minors'' (roughly, minors that have both a large clique-minor and a small restriction that is far from being graphic), but
it is not straightforward to find useful minors that are sufficiently connected. 
Perhaps the main contribution of this paper is a technical result, Theorem~\ref{maintech2}, that resolves this issue.
We anticipate that this result will prove useful for determining growth-rate functions of other quadratically dense classes: for example, the \emph{golden-mean matroids} representable over $\GF(4)$ and $\GF(5)$, which are conjectured by Archer [\ref{archer}] to have a growth-rate function of $\binom{n+3}{2} - 5$ for all $n \ge 4$. 

\subsection*{Unavoidable Minors}

For each integer $n \ge 4$, let $D_n$ denote the binary vertex-edge incidence matrix of $K_n$, and let $M_n^{\smallsq}$ denote the matroid $M(D_n | v)$, where $v$ is a binary column vector with exactly four nonzero entries. For $n \ge 3$, let $M_n^{\smalltri}$ denote the principal extension of a triangle in $M(K_n)$ (that is, the matroid formed by freely adding a point to the closure of a triangle of $M(K_n)$), and let $M_n^{\circ}$ denote the free extension of $M(K_n)$. We also prove the following theorem: 

\begin{theorem}\label{unavoidable}	Let $m,n$ be integers so that $m \ge 4$ and $n \ge 2m^2$. If $M$ is a nongraphic single-element extension of $M(K_n)$, then $M$ has a minor isomorphic to $M_m^{\smallsq},M_m^{\smalltri}$, or $M_m^{\circ}$. 
\end{theorem}

This gives us a stronger version of Theorem~\ref{main}. Let $\cG^{\smallsq}$ denote the closure under minors of the set $\{M_n^{\smallsq}: n \ge 4\}$, and define $\cG^{\smalltri}$ and $\cG^{\circ}$ similarly. These three classes are all have natural characterisations and easily determined growth rate functions (proofs of these statements and the relevant definitions appear in Section~\ref{extsection}):
\begin{itemize}
\item $\cG^{\smallsq}$ has growth-rate function $h_{\cG^{\smallsq}}(n) = \binom{n+2}{2}-3$ for all $n \ge 2$, and is the class of even-cycle matroids represented by a signed graph with a blocking pair. 
\item $\cG^{\smalltri}$ has growth-rate function $h_{\cG^{\smalltri}}(n) = \binom{n+2}{2}-2$ for all $n \ge 2$, and is the class of signed-graphic matroids having a signed-graph representation $(G,W)$ so that $G$ has a vertex $v$ incident with all non-loop edges in $W$. 
\item $\cG^{\circ}$ has growth-rate function $h_{\cG^{\circ}}(n) = \binom{n+2}{2}$ for all $n \ge 2$, and is the union of the class of graphic matroids and the class of truncations of graphic matroids. 
\end{itemize}
 %[\ref{zaslav}].) 

 Theorem~\ref{main} combined with Theorem~\ref{unavoidable} gives the following: 
 
 \begin{theorem}\label{strongermain}
Let $\cM$ be a quadratically dense minor-closed class of matroids. Either
\begin{enumerate}
\item $h_{\cM}(n) \approx \binom{n+1}{2}$, or
\item\label{stronger2} $\cM$ contains $\mcfree,\mcsigned$, or $\mcevencycle$.
\end{enumerate}
\end{theorem}

Theorem~\ref{maingap} is a consequence of the above theorem and the growth-rate functions stated for $\cG^{\circ}, \cG^{\smalltri}$ and $\cG^{\smallsq}$. 
%We obtain Theorem~\ref{maingap} from the above by computing the growth-rate functions of the three classes; these follow easily from our characterisations:

%\begin{lemma}
%For all $n \ge 2$, 
%\begin{itemize}
%\item $h_{\cG^{\smallsq}}(n) = \binom{n+2}{2}-3$,
%\item $h_{\cG^{\smalltri}}(n) = \binom{n+2}{2}-2$, and
%\item $h_{\cG^{\circ}}(n) = \binom{n+2}{2}$.
%\end{itemize}
%\end{lemma}

The three classes in Theorem~\ref{strongermain} are not all representable over all finite fields; this allows us to obtain stronger statements when every matroid in $\cM$ is representable over some fixed finite field. For any such field $\bF$, the co-line $U_{\ell,\ell+2}$ is not $\bF$-representable but is the truncation of the circuit $U_{\ell+1,\ell+2}$ so is in $\cG^{\circ}$. Therefore not every matroid in $\cG^{\circ}$ is $\bF$-representable. 

Note that $U_{2,4} \cong M_3^{\smalltri} \in \mcsigned$. Thus if outcome (\ref{stronger2}) of Theorem~\ref{strongermain} holds for a class $\cM$ of binary matroids, we have $\mcevencycle \subseteq \cM$. By the Growth-rate Theorem, this gives the following:

\begin{corollary}\label{binary}
If $\cM$ is a minor-closed class of binary matroids, then $h_{\cM}(n) \approx \binom{n+1}{2}$ if and only if $\cM$ contains all graphic matroids but not all matroids in $\cG^{\smallsq}$. 
\end{corollary}

An example of a nongraphic matroid in $\cG^{\smallsq}$ is the rank-$4$ matroid $N_{12}$ formed by deleting a three-element independent set from $\PG(3,2)$. Indeed, by the characterisation stated earlier, the simple rank-$4$ matroids in $\cG^{\smallsq}$ are exactly the rank-$4$ restrictions of $N_{12}$. By the above corollary, excluding any nongraphic matroid $N \in \cG^{\smallsq}$ as a minor from the class of binary matroids gives a class $\cM$ with $h_{\cM}(n) \cong \binom{n+1}{2}$, so this holds whenever $N$ is a nongraphic restriction of $N_{12}$. These restrictions include $\PG(2,2),\AG(3,2)$ and the rank-$4$ binary spike, so the corollary implies (for large $n$) growth-rate results for excluding these matroids proved respectively by Heller [\ref{heller}], Kung et al. [\ref{kmpr}] and McGuinness [\ref{mcguinness}]. The corollary also resolves a question posed in [\ref{kmpr}] of which simple rank-$4$ matroids, when excluded as a minor from the class of binary matroids, give a class with eventual growth-rate function $\binom{n+1}{2}$; they are exactly the nongraphic rank-$4$ restrictions of $N_{12}$.

Note that $F_7 \cong M_4^{\smallsq} \in \cG^{\smallsq}$, so if a minor-closed class $\cM$ contains only matroids representable over some fixed finite field of characteristic other then $2$, then $\mcevencycle\not\subseteq\cM$. Thus we have the following:

\begin{corollary}\label{oddchar}
Let $q$ be an odd prime power. If $\cM$ is a minor-closed class of $\GF(q)$-representable matroids, then $h_{\cM}(n) \approx \binom{n+1}{2}$ if and only if $\cM$ contains all graphic matroids but not all matroids in $\mcsigned$. 
\end{corollary}

By considering some well-known matroids in $\cG^{\smallsq}$, $\cG^{\smalltri}$, and $\cG^{\circ}$
we can get other interesting applications of  Theorem~\ref{strongermain}.
For example, for each $r\ge 2$, the whirl $\mathcal{W}^r$ is contained in $\cG^{\smalltri}$.
Moreover, $\cG^{\smallsq}$ contains the Fano matroid $F_7$ and $\cG^{\circ}$ contains the uniform matroid $U_{r,r+2}$.
Thus we obtain the following result:

\begin{corollary}\label{excl}
If $r \ge 2$ and $\cM$ is the class of matroids with no minor isomorphic to $U_{2,r+2},U_{r,r+2},\mathcal{W}^r$ or $F_7$, then $h_{\cM}(n) \approx \binom{n+1}{2}$. 
\end{corollary}

For each $r$, the free rank-$r$ spike $\Lambda_r$ is the truncation of $M(K_{2,r})$, so $\Lambda^r \in \cG^{\circ}$, and $U_{r,r+2}$ can also be  replaced by $\Lambda_r$ in the above theorem. 

For an odd-sized finite field $\GF(q)$ and $r \ge q$, all matroids but $\mathcal{W}^r$ in Corollary~\ref{excl} are not $\GF(q)$-representable, giving something simpler:

\begin{corollary}
If $q$ is an odd prime power, $r \ge 2$ is an integer, and $\cM$ is the class of $\GF(q)$-representable matroids with no $\mathcal{W}^r$-minor, then $h_{\cM}(n) \approx \binom{n+1}{2}$. 
\end{corollary}

\section{Preliminaries}

We use the notation of Oxley [\ref{oxley}]. A rank-$1$ flat is a \emph{point} and a rank-$2$ flat is a \emph{line}. Additionally, we write $|M|$ for $|E(M)|$ and $\elem(M)$ for $|\si(M)|$, the number of points in $M$. For an integer $\ell \ge 2$, we write $\cU(\ell)$ for the class of matroids with no $U_{2,\ell+2}$-minor. Finally, in everything that follows we will abbreviate the term `single-element extension' simply to `extension'. 

We require a theorem of Kung [\ref{kung91}] that bounds the number of points in a matroid in $\cU(\ell)$. 	
\begin{theorem}\label{kung}
If $\ell \ge 2$ and $M \in \cU(\ell)$ then $\elem(M) \le \frac{\ell^{r(M)}-1}{\ell-1}$. 
\end{theorem}

We will use this theorem freely, usually with the weaker bound $\elem(M) < \ell^{r(M)}$ for convenience of calculation. The next result we need is a constituent of the Growth-rate Theorem that shows that any matroid in $\cU(\ell)$ with sufficiently large `linear' density has a large clique as a minor.  

\begin{theorem}\label{lgrt}
There is a function $\alpha_{\ref{lgrt}}: \bZ^2 \to \bR$ so that, for all $n,\ell \in \bZ^+$, if $M \in \cU(\ell)$ and $\elem(M) > \alpha_{\ref{lgrt}}(n,\ell) r(M)$, then $M$ has an $M(K_{n+1})$-minor. 
\end{theorem}

We also need a special case of the Erd\H{o}s-Stone theorem [\ref{esthm}]:

\begin{theorem}\label{es}
There is a function $f_{\ref{es}}(\alpha,m) : \bR \times \bZ \to \bZ$ so that, for all $\alpha \in \bR$, $n \in \bZ$ with $\alpha > 0$ and $n \ge 1$,  every simple graph $G$ with $|V(G)| \ge f_{\ref{es}}(\alpha,m)$ and $|E(G)| \ge \alpha |V(G)|^2$ has a $K_{m,m}$-subgraph.
\end{theorem}

Finally, we require a version of Tutte's Linking Theorem proved by Geelen, Gerards and Whittle [\ref{ggw06}],
for which we recall some standard notation. For disjoint sets $X,Y \subseteq E$ 
in a matroid $M=(E,r)$, we let $\lambda_M(X) = r(X) + r(E-X) - r(E)$ and we let 
$\kappa_M(X,Y)$ denote the minimum of $\lambda_M(Z)$ taken over all sets $Z$ with
$X\subseteq Z\subseteq E-Y$. 

\begin{theorem}[Tutte's Linking Theorem]\label{linking}
If $M$ is a matroid and $X,Y \subseteq E(M)$ are disjoint, then $M$ has a minor $N$ with $E(N) = X \cup Y$ so that $N|X = M|X$ and $N|Y = M|Y$ while $\lambda_N(X) = \kappa_M(X,Y)$. 
\end{theorem}

\section{Unavoidable Minors}
In this section we prove Theorem~\ref{unavoidable}. 
We first need some basic facts about extensions; all follow from material in [\ref{oxley}], Section 7.2. A pair of flats $F_1,F_2$ of a matroid $M$ is a \emph{modular pair} in $M$ if $r_M(F_1) + r_M(F_2) = r_M(F_1 \cup F_2) + r_M(F_1 \cap F_2)$. A flat is \emph{modular} in $M$ if it forms a modular pair with every flat of $M$. 
If $M \cong M(K_n)$, then a flat $F$ of $M$ is modular if and only if $M|F$ is connected.

We now consider extensions of cliques. Our first lemma deals with extensions where the new point is placed in some connected flat of rank much less than $r(M)$. 
\begin{lemma}\label{smallcliqueext}
Let $m \ge 4$ be an integer. If $M$ is a nongraphic extension of a clique by an element $e$, and $e \in \cl_M(F)$ for some modular flat $F$ of $M \del e$ such that $r(M) - r_M(F) \ge m-2$, then $M$ has a minor isomorphic to $\triangleext{m}$ or $\fanoext{m+1}$. 
\end{lemma}
\begin{proof}
We may assume that $M$ is minor-minimal subject to the hypotheses; 
let $F$ be the minimal modular flat of $M\del e$ with $e\in\cl_M(F)$. Let $r = r(M)$. 
Note that $M$ is the modular sum (also known as \emph{generalised parallel connection})
of $M\del e\cong M(K_{r+1})$ and $M|(F\cup \{e\})$, so $M$ is uniquely determined by
$M|(F\cup \{e\})$ and $r$.

By the minor-minimality of $M$, each element of $F$ is on a line containing $e$ and at least one other element.
Since each pair of elements of $M(K_{r+1})$ is spanned by a modular flat of rank at most $3$,
we have that $r(F)\le 3$.
Now it is easy to see that either $M|(F\cup\{e\})\cong U_{2,4}$ (in which case $M\cong \triangleext{m}$) or
$M|(F\cup\{e\})\cong F_7$ (in which case $M\cong \fanoext{m+1}$).
\end{proof}

We now restate and prove Theorem~\ref{unavoidable}. 
\begin{theorem}\label{cliqueext}
Let $m,n$ be integers such that $m \ge 4$ and $n \ge 2m^2$. If $M$ is a nongraphic extension of a rank-$n$ clique, then $M$ has a minor isomorphic to $\freeext{m}$, $\triangleext{m}$ or $\fanoext{m}$.
\end{theorem}
\begin{proof}
Let $G \cong K_{n+1}$ and let $e \in E(M)$ be such that $M \del e \cong M(G)$. Let $F$ be a minimal flat of $M \del e$ such that $e \in \cl_M(F)$. Since $M|F$ has at most $r_M(F)$ components and any two such components are joined by an edge of $G$, there is a flat $\hat{F}$ of $M$ containing $F$ such that $M|\hat{F}$ is connected and $r_M(\hat{F}) < 2r_M(F)$. If $r_M(\hat{F}) \le 2m(m-1)$, then $n - r_M(\hat{F}) \ge m$, and $M$ has a $\triangleext{m}$-minor or a $\fanoext{m}$-minor by Lemma~\ref{smallcliqueext}. We may thus assume that $r_M(\hat{F}) > 2m(m-1)$ and so $r_M(F) > m(m-1)$. 

 Since $F$ is a flat of $M(G)$ and $G \cong K_{n+1}$, there are vertex-disjoint complete subgraphs $C_1,C_2, \dotsc, C_t$ of $G$ such that $|V(C_i)| \ge 2$ for each $i$ and $F = E(C_1) \cup \dotsc \cup E(C_t)$; let $F_i = E(C_i)$ for each $i$. Note that $r_M(F) = \sum_{i = 1}^t r_M(F_i)$. Let $G'$ be the complete subgraph of $G$ with vertex set $\cup_{i = 1}^t V(C_i)$, so $r_M(E(G')) = r_M(F) + t - 1$.  

If $r_M(F_i) \ge m-1$ for some $i$, then let $B$ be a basis for $F$ containing an $(m-1)$-element independent set $I \subseteq F_i$. Now $\si((M|F)/(B-I)) \cong \freeext{m}$, giving the lemma. Otherwise $r_M(F_i) < m-1$ for each $i$, so $r_M(F) < t(m-1)$. Therefore $t(m-1) > m(m-1)$ and $t > m$.

 Let $f$ be an edge of $G'$ with one end in $C_1$ and the other in $C_2$. Let $M' = (M|E(G')) \con (\{f\} \cup (F - (F_1 \cup F_2)))$. Let $F' = \cl_{M' \del e}(F_1 \cup F_2)$. Now $\si(M' \del e)$ is a clique. Moreover, $M'|F'$ is connected, has rank at least $2$, and $F'$ is a minimal flat of $M' \del e$ spanning $e$ in $M'$. Since $r(M') = r(M|E(G')) -1 - r_M(F) + r_M(F_1 \cup F_2) = r_M(F') + t-2$ and $t > m$, Lemma~\ref{smallcliqueext} implies that $\si(M')$ has an $\triangleext{m}$-minor or an $\fanoext{m}$-minor.	
\end{proof}

%Let $G = (V,E)$ be a graph, and let $W \subseteq E$. If $A \in \GF(2)^{V \times E}$ is the binary incidence matrix of $G$ and $w \in \GF(2)^E$ is the characteristic vector of $W$, then the \emph{even cycle matroid} $M^{\smallsq}(G,W)$ represented by $(G,W)$ is the binary matroid $M\binom{w}{A}$. It is easy to see that $M(G,W) \del e = M(G \del e,W - \{e\})$ for all $e \in E$, that $M(G,W) \con e = M(G \con e,W)$ for all $e \in E - W$, and that $M(G,W) = M(G,W \oplus \delta_G(v))$ for every vertex $v$ of $G$, where $\oplus$ denotes symmetric difference and $\delta_G(v)$ denotes the set of edges incident with $v$. A \emph{blocking pair} of $(G,W)$ is a set $\{u,v\}$ of two vertices of $G$ so that every edge in $W$ is incident with $u$ or $v$ (some authors use this definition slightly differently, also insisting that no single vertex has this property). It follows from the above facts that the class of even cycle matroids is minor-closed, as is the class of even cycle matroids having a graph representation $(G,W)$ with a blocking pair. 

\section{Unavoidable classes}\label{extsection}

In this section we give proofs of the characterisations and growth-rate functions of the classes $\cG^{\smallsq}, \cG^{\smalltri}$ and $\cG^{\circ}$ claimed in the introduction. Our discussion of the even-cycle and signed-graphic matroids is quite terse; these well-known classes are treated thoroughly in [\ref{zaslav}]. 

An \emph{even-cycle matroid} is a binary matroid of the form $M = M\binom{w}{D}$, where $D \in \GF(2)^{V \times E}$ is the vertex-edge incidence matrix of a graph $G = (V,E)$ and $w \in \GF(2)^E$ is the characteristic vector of a set $W \subseteq E$. The pair $(G,W)$ is an \emph{even-cycle representation} of $M$. A \emph{blocking pair} of $(G,W)$ is a pair of vertices $u,v$ of $G$ so that every edge in $W$ is incident with either $u$ or $v$ (some authors use this term more restrictively and insist that no single vertex has this property).

\begin{lemma}\label{characterisesquare}
	$\cG^{\smallsq}$ has growth-rate function $h_{\cG^{\smallsq}}(n) = \binom{n+2}{2} - 3$ for all $n \ge 2$, and is exactly the class of even-cycle matroids having an even-cycle representation with a blocking pair. 
\end{lemma}
\begin{proof}
	For each $n \ge 2$, let $N_n^{\smallsq}$ be the rank-$n$ matroid obtained from $M_{n+2}^{\smallsq}$ by contracting the extension point. It is easy to check that every rank-$r$ matroid of the form $\si(M_n^{\smallsq} \con C)$ is isomorphic to one of $M(K_{r+1}), M_{r+1}^{\smallsq}$ or $\si(N_r^{\smallsq})$. Moreover, both $M(K_{r+1})$ and $M_{r+1}^{\smallsq}$ are restrictions of $N_r^{\smallsq}$. Therefore the simple rank-$n$ matroids in $\cG^{\smallsq}$ are exactly the simple rank-$n$ restrictions of $N_n^{\smallsq}$. Moreover, since the class of graphic matroids is closed under parallel extension and adding loops, the class $\cG^{\smallsq}$ contains every matroid whose simplification is isomorphic to $N_n^{\smallsq}$, so the rank-$n$ matroids in $\cG^{\smallsq}$ are exactly those whose simplification is isomorphic to a restriction of $N_n^{\smallsq}$. Since $|\si(N_n^{\smallsq})| = \binom{n+2}{2}-3$, the claimed growth-rate function for $\cG^{\smallsq}$ follows.
	
	Furthermore, by considering a binary matrix representation of $M_{n+2}^{\smallsq}$, we see that $N_n^{\smallsq}$ has a graph representation $(G,W)$, where $G \del W \cong K_n$ and $W$ consists of a loop, together with an edge between $x$ and $y$ for all distinct $x,y \in V(G)$ with $\{x,y\} \cap \{u,v\} \ne \varnothing$. It follows that every matroid whose simplification is isomorphic to a restriction of $N_n^{\smallsq}$ has a graph representation with a blocking pair, and every rank-$n$ even-cycle matroid having a graph representation with a blocking pair has simplification isomorphic to a restriction of $N_n^{\smallsq}$.	The lemma follows.  
\end{proof}

A \emph{signed-graphic} matroid is one represented by a $\GF(3)$-matrix in which each column has at most two nonzero entries. Let $G = (V,E)$ be a graph, let $W \subseteq E$, and for each $v \in V$ let $b_v \in \GF(3)^V$ denote the standard basis vector corresponding to $v$. Let $M$ be the ternary matroid on ground set $E$ with matrix representation $A \in \GF(3)^{V \times E}$, where for each edge $e = uv \in E$ we have $A_e = b_u + b_v$ if $e \in W$ and $A_e = b_u - b_v$ otherwise. (The definition of $A$ involves a choice of orientation for each edge but this does not affect $M$.) We say that $(G,W)$ is a \emph{signed-graph representation of $M$}.
\begin{lemma}\label{characterisetriangle}
	$\cG^{\smalltri}$ has growth rate function $h_{\cG^{\smalltri}}(n) = \binom{n+2}{2}-2$ for all $n \ge 2$, and is exactly the class of signed-graphic matroids having a graph representation $(G,W)$ so that there is some $v \in V(G)$ incident with every negative nonloop edge.  
\end{lemma}
\begin{proof}
	Note that $M_{n+2}^{\smalltri} \cong M(I_{n+1}| D_{n+1}' | w)$, where $I_{n+1}$ is the ternary identity matrix, $D_{n+1}'$ is some ternary signed incidence matrix of $K_{n+1}$, and $w \in \GF(3)^{n+1}$ is the sum of the first two standard basis vectors. Let $N_n^{\smalltri}$ denote the matroid obtained from $M_{n+2}^{\smalltri}$ by contracting the element corresponding to $w$.
	
	Similarly to the previous lemma, every rank-$r$ matroid of the form $\si(M_n^{\smalltri} \con C)$ is isomorphic to one of $M(K_{r+1}), M_{r+1}^{\smalltri}$, or $\si(N_r^{\smalltri})$. Moreover, both $M(K_{r+1})$ and $M_{r+1}^{\smalltri}$ are restrictions of $N_r^{\smalltri}$, so the simple rank-$n$ matroids in $\cG^{\smalltri}$ are exactly the simple rank-$n$ restrictions of $N_n^{\smalltri}$. Just as in the previous lemma, we also have that the rank-$n$ matroids in $\cG^{\smalltri}$ are exactly those whose simplification is isomorphic to a rank-$n$ restriction of $N_n^{\smalltri}$. Since $|\si(N_n^{\smalltri})| = \binom{n+2}{2}-2$, the claimed growth-rate function for $\cG^{\smalltri}$ follows.  
	
	By considering the $\GF(3)$-representation of $M_{n+2}^{\smalltri}$ given above, we see that $N_n^{\smalltri}$ is represented by some $\GF(3)$-matrix in which each column has at most two nonzero entries, and in which each column having both nonzero entries equal to $1$ has one such entry in the first row. Therefore $N_n^{\smalltri}$ has a signed-graph representation of the claimed form, and moreover the class of all simple rank-$n$ matroids having such a representation is exactly the class of simple rank-$n$ restrictions of $N_n^{\smalltri}$. Similarly to the previous lemma, the characterisation of $\cG^{\smalltri}$ follows. 
\end{proof}

Recall that the \emph{truncation} $T(M)$ of a matroid $M$ is the matroid with ground set $E(M)$ constructed by freely extending $M$ by a point $e$, then contracting $e$. 

\begin{lemma}
	The class $\cG^{\circ}$ has growth-rate function $h_{\cG^{\circ}}(n) = \binom{n+2}{2}$ for all $n \ge 2$, and is exactly the union of the class of graphic matroids and the class of truncations of graphic matroids. 
\end{lemma}
\begin{proof}
	Let $\cG$ and $\cG_T$ denote the classes of graphic matroids and truncations of graphic matroids respectively. It is clear that $h_{\cG \cup \cG_T}(n) = \binom{n+2}{2}$ for all $n \ge 2$ and that $\cG \cup \cG_T \subseteq \cG^{\circ}$, so it suffices to show that $M_n^{\circ} \in \cG_T$ for all $n \ge 2$. Indeed, we have $M_n^{\circ} \cong T(M(G))$, where $G = (V,E)$ is the connected graph on $n+1$ vertices so that $G \del e \cong K_n$ for some $e \in E(G)$, since $T(M(G))|(E-\{e\}) \cong M(G)|(E - \{e\}) \cong M(K_n)$, and the point $e$ is freely placed in the span of $E-\{e\}$ in $T(M(G))$.
	\end{proof}

\section{Complete Bipartite Graphs}

In this section we show that, for very large $n$, a matroid obtained from $M(K_{n,n})$ by a bounded number of coextension-deletion operations contains an $M(K_{m,m})$-restriction for some large $m$. 

\begin{lemma}\label{singleliftclique}
There is a function $f_{\ref{singleliftclique}}: \bZ^2 \to \bZ$ so that, 
for each $\ell,m,n \in \bZ$ with $\ell \ge 2$, $m \ge 1$, and  $n \ge f_{\ref{singleliftclique}}(\ell,m)$,
if $e$ is an element of a matroid $M\in\cU(\ell)$ such that 
$M \con e \cong M(K_{n,n})$,
then $M \del e$ has a $K_{m,m}$-restriction. 
\end{lemma}
\begin{proof}
Set $f_{\ref{singleliftclique}}(\ell,m) = f_{\ref{es}}\left(\tfrac{1}{8\ell+8},m\right)$. 
Let $n\ge f_{\ref{singleliftclique}}(\ell,m)$, let $G\cong K_{n,n}$, and let
$e$ be an element of a matroid $M\in\cU(\ell)$ such that $M \con e = M(G)$.
Let $T_1$ and $T_2$ be vertex-disjoint copies of $K_{1,n-1}$ in $G$,
let $T= E(T_1)\cup E(T_2)$, and let $F$ denote the set of edges of $G$ with an end
in $V(T_1)$ and an end in $V(T_2)$. 
Note that $|F| > (n-1)^2\ge \frac{n^2}{2}$ and that $F$ 
is the set of nonloop elements of the rank-$1$ matroid $M \con (\{e\} \cup T)$. 

Now  $M \con T\in \cU(\ell)$ and $r(M\con T)\ge 2$, so  there is some set $F' \subseteq F$ 
such that $|F'| \ge \tfrac{1}{\ell+1}|F|$ and $F'$ is contained in a parallel class of $M \con T$. Therefore $F'$ has rank $1$ in both $M \con T$ and $M \con (T \cup \{e\})$, so $e \notin \cl_{M \con T}(F')$ and $e \notin \cl_M(F')$. Thus $M|F' = (M \con e)|F'$. But $G[F']$ is a simple graph with $2n$ vertices and at least $\tfrac{(n-1)^2}{\ell+1} \ge \tfrac{1}{8\ell+8}(2n)^2$ edges; since $n \ge f_{\ref{es}}\left(\tfrac{1}{8\ell+8},m\right)$, it follows by Theorem~\ref{es} that $G[F']$ has a $K_{m,m}$-subgraph, so $(M \con e)|F' = M|F'$ has an $M(K_{m,m})$-restriction, as required. 
\end{proof}

\begin{lemma}\label{slogliftclique}
There is a function $f_{\ref{slogliftclique}}: \bZ^3 \to \bZ$ so that, 
for each $\ell,m,t,n \in \bZ$ with $\ell \ge 2$,  $m > t \ge 0$,
and $n \ge f_{\ref{slogliftclique}}(\ell,m,t)$,
%if $X$,$C$ and $K$ are sets in a matroid $M \in \cU(\ell)$ such that $C \subseteq X$, $\sqcap_M(X,K) \le t$ 
if $M \in \cU(\ell)$ and $C,X,K \subseteq E(M)$ satisfy $C \subseteq X$, $\sqcap_M(X,K) \le t$
and $(M \con C)|K \cong M(K_{n,n})$, then $M|(K-X)$ has an $M(K_{m,m})$-restriction. 
\end{lemma}
\begin{proof}
Let $\ell,m,t,n \in \bZ$ with $\ell \ge 2$ and  $m > t \ge 0$.
Now let $m' = \max(m+t+1,f_{\ref{singleliftclique}}(\ell,m))$ and 
define $f_{\ref{slogliftclique}}$ recursively by 
 $f_{\ref{slogliftclique}}(\ell,m,t) = f_{\ref{slogliftclique}}(\ell,m',t-1)$.
 
Let $n\ge f_{\ref{slogliftclique}}(\ell,m,t)$, let $M \in \cU(\ell)$, and let $C,X,K$ be subsets of $M$ such that $C \subseteq X$, $\sqcap_M(C,X) \le t$ and $(M \con C)|K \cong M(K_{n,n})$.
We may assume that $C$ is independent in $M$. Let $C_1$ be a maximal subset of $C$ that is skew to $K$ in $M$, and let $C_0 = C - C_1$. Now $(M \con C_0) |K = M|K$ and $C_1 \subseteq \cl_{M \con C_0}(K)$ by maximality. Moreover, $|C_1| \le \sqcap_M(X,K) \le t$. If $C_1 = \varnothing$ then $(M \con C)|K = M|K \cong M(K_{n,n})$ and $r_M(X \cap K) \le t$, so $M|K$ has an $M(K_{n-(t+1),n-(t+1)})$-restriction, giving the result since $n -(t+1) \ge m$. Otherwise, let $e \in C_1$ and let $M' = M \con C_0$. Since $e \in X \cap \cl_{M'}(K)$, we have \[\sqcap_{M' \con e}(X-e,K) \le \sqcap_{M'}(X,K)-1 \le t-1.\] 

Since $(M' \con C_1)|K \cong M(K_{n,n})$ and $n \ge f_{\ref{slogliftclique}}(\ell,m',t-1)$, applying the inductive hypothesis to $C_1 - e, X-e$ and $K$ in $M' \con e$ gives that $(M' \con e)|(K-(X-e))$ has an $M(K_{m',m'})$-restriction $R$. By Lemma~\ref{singleliftclique} applied to $M'|(\{e\} \cup E(R))$, the matroid $M'|E(R)$ has an $M(K_{m,m})$-restriction. Since $E(R) \subseteq K - X$ and $M'|E(R) = M|E(R)$, the lemma follows. 
\end{proof}

\section{Vertical Connectivity}

We now detail a somewhat elaborate connectivity reduction, showing that quadratically dense classes contain dense, highly vertically connected matroids with some additional structure. We expect this reduction to be of much more general use in determining growth-rate functions; we will invoke it in this paper just for $s = 4$. 

\begin{theorem}\label{maintech2}
Let $\cM$ be a quadratically dense minor-closed class of matroids and let $p(x)$ be a real quadratic polynomial with positive leading coefficient. If $h_{\cM}(n) > p(n)$ for infinitely many $n \in \bZ^+$, then for all integers $r,s \ge 1$ there exists $M \in \cM$ satisfying $\elem(M) > p(r(M))$ and $r(M) \ge r$ such that either
\begin{enumerate}
\item $M$ has an spanning clique restriction, or
\item $M$ is vertically $s$-connected and has an $s$-element independent set $S$ so that $\elem(M) - \elem(M \con e) > p(r(M)) - p(r(M)-1)$ for each $e \in S$.
\end{enumerate}
\end{theorem}

\begin{proof}
Let $\ell$ be an integer such that $U_{2,\ell+2} \notin \cM$. Let $\cQ$ be the set of all real quadratic polynomials $q$ such that $q$ has positive leading coefficient and $h_{\cM}(n) > q(n)$ for infinitely many $n \in \bZ^+$. Our first claim gives a weaker version of the theorem:

\begin{claim}
For each $q \in \cQ$ and $r,s \in \bZ^+$, there is a matroid $M \in \cM$ of rank at least $r$ such that $\elem(M) > q(r(M))$ and either
\begin{enumerate}[(a)]
\item\label{claim1a} $M$ has a spanning clique restriction, or 
\item\label{claim1b} $M$ has an $s$-element independent set $S$ such that each $e \in S$ satisfies $\elem(M) - \elem(M \con e) > q(r(M)) - q(r(M)-1)$.  
\end{enumerate}
\end{claim}
\begin{proof}[Proof of claim:]
Let $n_2 \ge r+1$ be an integer such that $q(x) - q(y) \ge  \ell^s$ for all real $x,y$ with $x \ge n_2$ and $x-1 \ge y \ge 0$. Let $n_1 = (s(s-1)+1)n_2$. Let $n_0$ be an integer such that $q(x) \ge  \alpha_{\ref{lgrt}}(n_1-1,\ell) x$ for all real $x \ge n_0$. 

Let $M_0 \in \cM$ satisfy $\elem(M_0) > q(r(M_0))$ and $r(M_0) \ge n_0$. By Theorem~\ref{lgrt} we know that $M_0$ has a $M(K_{n_1})$-minor $N_0$. Let $M_1$ be a minimal minor of $M_0$ such that $\elem(M_1) > q(r(M_1))$ and $N_0$ is a minor of $M_1$. Note that $r(M_1) \ge r(N_0) \ge r$. Let $C$ be an independent set in $M_1$ so that $N_0$ is a spanning restriction of $M_1 \con C$. By minimality, we have $\elem(M_1) - \elem(M_1 \con e) > q(r(M_1))-q(r(M_1)-1)$ for each $e \in C$. If $|C| \ge s$ then $M_1$ and $C$ satisfy (\ref{claim1b}), so we may assume that $|C| < s$.  

Let $i \ge 0$ be minimal so that there is a minor $M_2$ of $M_1$ for which
\begin{enumerate}[(i)]
\item\label{minia}
$\elem(M_2) > q(r(M_2))$, and 
\item\label{minib} there exists $X \subseteq E(M_2)$ such that $r_{M_2}(X) \le i$ and $M_2 \con X$ has an $M(K_{(is+1)n_2})$-restriction $N_2$. 
\end{enumerate}
(Note that $(i,M_2) = (s-1,M_1)$ is a candidate, so this choice is well-defined.) We consider two cases depending on whether $i = 0$. 

Suppose that $i>0$ and let $Y_1,Y_2,\dotsc,Y_s,Z$ be mutually skew sets in $N_2$ so that $N_2|Y_i \cong M(K_{n_1})$ for each $i \in \{1, \dotsc, s\}$ and $N_2|Z \cong M(K_{((i-1)s + 1)n_2})$; these sets can be chosen to correspond to vertex-disjoint cliques in the clique underlying $N_2$. If $M_2|Y_j = N_2|Y_j$ for some $j \in \{1, \dotsc, s\}$, then $M_2$ has an $M(K_{r+1})$-restriction so satisfies (\ref{minia}) and (\ref{minib}) for $i = 0$, contradicting the mimimality of $i$. Thus, $M_2|Y_j \ne N_2|Y_j$ for each $j$, implying that $\sqcap_{M_2}(Y_j,X) > 0$ and $r_{M_2 \con Y_j}(X) \le r_{M_2}(X) -1 \le i - 1$ for each $j$. Let $Y = Y_1 \cup \dotsc \cup Y_s$ and let $J$ be a maximal subset of $Y$ such that $\elem(M_2 \con J) > q(r(M_2 \con J))$. Let $M_3 = M_2 \con J$. 
If $Y_j \subseteq J$ for some $j$, then $r_{M_3}(X) \le i-1$ and $(M_3 \con X)|Z = N_2|Z \cong M(K_{((i-1)s+1)n_2})$, contradicting the minimality of $i$. Therefore $Y - J$ contains a transversal $T$ of $(Y_1, \dotsc, Y_s)$. $T$ is an $s$-element independent set of $N_2 \con J$ and therefore of $M_2 \con J$. Moreover, by maximality of $J$, each $e \in T$ satisfies $\elem(M_3) - \elem(M_3 \con e) > q(r(M_3)) - q(r(M_3)-1)$. Since $r(M_3) \ge r(N_2|Z) \ge n_2-1 \ge r$, now (\ref{claim1b}) holds for $M_3$ and $T$. 

Now suppose that $i = 0$. Then $N_2$ is an $M(K_{r+1})$-restriction of $M_2$. Let $M_4$ be a minimal minor of $M_2$ such that $\elem(M_4) > q(r(M_4))$ and $N_2$ is a restriction of $M_4$. If $N_2$ is spanning in $M_4$ then (\ref{claim1a}) holds. Otherwise, by minimality we have $\elem(M_4|\cl_{M_4}(E(N_2))) \le q(r(N_2))$, so since $r(M_4) \ge n_2$ we have
\begin{align*} 
\elem(M_4 \del \cl_{M_4}(E(N_2))) &> q(r(M_4)) - q(r(N_2)) \\
								  &\ge q(r(M_4))-q(r(M_4)-1)\\
								  &\ge \ell^s.
\end{align*}
Therefore there is an $s$-element independent set $S$ of $M_4$ that is disjoint from $\cl_{M_4}(E(N_2))$. Since $N_2$ is a restriction of $M_4 \con e$ for each $e \in S$, it follows that $M_4$ and $S$ satisfy (\ref{claim1b}). 	 
\end{proof}	

Suppose that the theorem does not hold for some positive integers $s_0$ and $r_0$. 
Let $a,b,c\in \bR$ such that $p(x) = ax^2 + bx + c$; thus $a>0$. 

\begin{claim}
The quadratic polynomial $p(x) + \nu x$ is in $\cQ$ for all $\nu \in \bR$. 
\end{claim}
\begin{proof}[Proof of claim:]
Suppose not; then there exists some $\nu \ge 0$ for which $p(x) + \nu x \in \cQ$ but $p(x) + (\nu + a)x \notin \cQ$. Let $r_1$ be an integer so that 
\begin{align}\label{defr2}
(2s_0+1)a(x+y) + s_0|\nu+b| + c - as_0^2 \le 2a x y
\end{align}
for all real $x,y \ge r_1$, and 
\begin{align}\label{defr1b}
h_{\cM}(n) \le p(n) + (\nu + a)n \text{ for every integer $n \ge r_1$.}
\end{align} 
Let $r_2 \ge \max(r_0,2r_1)$ be an integer so that
\begin{align}\label{defr1a}
p(x)-p(x-1) > ax + \ell^{r_1} \text{ for all real $x \ge r_2$.}
\end{align}
By the first claim, there exists $M \in \cM$ of rank at least $r_2$, such that $\elem(M) > p(r(M)) + \nu r(M)$ and either $M$ has a spanning clique or there is an $s_0$-element independent set $S$ of $M$ so that 
\begin{align*}
\elem(M) - \elem(M \con e) &> p(r(M)) - p(r(M)-1) + \nu
\end{align*}
for each $e \in S$. Since $\nu \ge 0$ and the theorem does not hold for $s_0$ and $r_0$, the matroid $M$ is not vertically $s_0$-connected. We may assume that $M$ is simple; let $(A,B)$ be a partition of $E(M)$ so that $r_M(A) \le r_M(B) < r(M)$ and $r_M(A) + r_M(B) - r(M) < s_0-1$. Let $r = r(M), r_A = r_M(A)$ and $r_B = r_M(B)$.
	
If $r_A < r_1$, then $|A| < \ell^{r_1}$, so since $r \ge r_2$, by (\ref{defr1a}) we have
\begin{align*}
|B| = |M| - |A| &> p(r) + \nu r - \ell^{r_1} \\ 
			   &> p(r-1) + (\nu+a)r\\
			   &\ge p(r_B) + (\nu +a)r_B,
\end{align*}
contradicting (\ref{defr1b}), since $r_B \ge r-r_A \ge r_2-r_1 \ge r_1$. So we have $r_B \ge r_A \ge r_1$. Therefore, using (\ref{defr1b}) we have
\begin{align*}
p(r) + \nu r < |A| + |B| 
				  \le p(r_A) + p(r_B) + (\nu + a)(r_A + r_B).
\end{align*}
Using $r_A + r_B < r + s_0$, expanding $p(x) = ax^2 + bx + c$ and simplifying, we have 
\[(2s_0+1)a(r_A + r_B) + {s_0|\nu + b| + c-as_0^2}> 2a r_Ar_B.\]
Since $r_B \ge r_A \ge r_1$, this contradicts (\ref{defr2}). 
\end{proof}

 Let $\alpha >0$ be such that $h_{\cM}(n) \le \alpha p(n)$ for all $n \in \bZ^+$. Let $n_1$ be an integer so that	 $p(x) \ge p(x-1) \ge 0$ for all real $x \ge n_1$ and 
\[a(\alpha + 2s_0)(x+y) + ((\alpha+1)b+\alpha |c|)s_0 + c-as_0^2 \le 2axy\]
for all real $x,y \ge n_1$. 
Let $\nu = \max(-b,\ell^{n_1},\ell^{n_1} - \min_{x \in \bR} p(x))$. 

Let $M \in \cM$ be minor-minimal such that $r(M) > 0$ and $\elem(M) > p(r(M)) + \nu r(M)$. (Such a matroid exists by the previous claim.) Note that $M$ is simple; let $r = r(M)$. We have $\elem(M) > \nu + p(r(M)) \ge \ell^{n_1}$, so $r(M) \ge n_1$. 

For each $e \in E(M)$, minimality of $M$ implies that 
\[\elem(M) - \elem(M \con e) > p(r)-p(r-1) + \nu.\]
This expression exceeds $p(r)-p(r-1)$, and $r(M) \ge n_1 \ge \max(r_0,s_0)$; since the lemma does not hold for $s_0$ and $r_0$, we know that $M$ is not vertically $s_0$-connected. Let $(A,B)$ be a partition of $E(M)$ so that $r_M(A) \le r_M(B) < r$ and $r_M(A) + r_M(B) < r(M) + s_0-1$. Let $r_A = r_M(A)$, $r_B = r_M(B)$.	

We first argue that $r_A \ge n_1$. If not, then $|A| < \ell^{n}$, so we have 
\begin{align*}
|B| &= |M| - |A| \\
&> p(r) + \nu r - \ell^{n_1}\\
&\ge p(r-1) + \nu(r-1)\\
&\ge p(r_B) + \nu r_B,
\end{align*}
which contradicts minimality. Next, since $r \ge n_1$ we have $p(r) \ge 0$ and so $\nu r < |M| \le \alpha p(r)$; since $r \ge 1$ this implies that 
\[\nu \le \alpha(ar + b + \tfrac{c}{r}) \le \alpha(a(r_A + r_B) + b + |c|).\] 
Now
\begin{align*}
p(r_A) + \nu r_A + p(r_B) + \nu r_B \ge |M| > p(r) + \nu r. 
\end{align*}
Using $r_A + r_B < r + s_0$ and $\nu + b \ge 0$, expanding $p$ as earlier gives
\[s_0(\nu + b) + c - as_0^2 + 2as_0(r_A+r_B) > 2r_Ar_B.\]
Combining this with our estimate for $\nu$, we have
\[a(\alpha + 2s_0)(r_A + r_B) + ((\alpha+1)b + \alpha |c|)s_0 + c - as_0^2 > 2ar_Ar_B,\]
contradicting $r_B \ge r_A \ge n_1$ and the definition of $n_1$. 	
\end{proof}

\section{Spikes}

A point of a matroid $M$ whose contraction substantially reduces the number of points of $M$ often gives rise to a \emph{spike}. This structure is well-known and its definitions vary slightly across the literature; here we give a definition convenient for extremal arguments that allows for any positive number of `tips' but no `co-tips'. 

A \emph{spike} is a matroid $S$ with ground set $E(S) = X \cup Y \cup T$, where $X,Y,T$ are disjoint sets so that $T$ is a nonempty parallel class, $S|(X\cup Y)$ is simple, and $X$ and $Y$ are circuits of $S \con T$ so that each line of $S$ containing $T$ contains exactly one element of each of $X$ and $Y$. Note that $|X| = |Y|$. An element in $T$ is a \emph{tip} of $S$. 

It is clear from this definition that if $r(S) > 3$ then contracting a non-tip element yields a rank-$(r(S)-1)$ spike. If $r(S) = 3$ then $S$ has three distinct three-point lines through its tip, so $\elem(S) = 7$ and thus $S$ is nongraphic; therefore all spikes of rank at least three are nongraphic. 

\begin{lemma}\label{breakspikeeasy}
If $S$ is a spike-restriction of a matroid $M$, and $e$ is a nonloop of $M$ not parallel to a tip of $S$, then there are spike-restrictions $S_1$ and $S_2$ of $M \con e$ such that $E(S) - \{e\} = E(S_1) \cup E(S_2)$. 
\end{lemma}
\begin{proof}
 If $e \notin \cl_M(E(S))$ or $e$ is parallel to an element of $E(S)$, then the result holds with $S_1 = S_2 = S$, so we may assume otherwise; we may also assume that $E(M) = E(S) \cup \{e\}$. Let $T,X,Y$ be sets as in the definition, and let $t \in T$. It suffices to show that $(M \con \{t,e\})|X$ is the union of two circuits. Since $X$ is a circuit of $M \con t$, we have $r_{(M \con t)^*}(X) = 1$, so $r_{(M \con t)^*}  (X \cup \{e\}) \le 2$ and so $r^*(M \con \{t,e\} | X) \le 2$. Every loopless matroid of rank at most $2$ is clearly the union of two cocircuits, so $(M \con \{t,e\})|X$ is the union of two circuits, as required. 
\end{proof}

\begin{lemma}\label{linkspikerestr}
Let $S$ be a spike-restriction of a matroid $M$. If $R$ is a restriction of $M \del E(S)$ satisfying $\kappa_M(E(S),E(R)) \ge 3$, then $M$ has a minor with $R$ as a spanning restriction and with a nongraphic spike-restriction.
\end{lemma}
\begin{proof}
Let $M'$ be a minimal minor of $M$ such that $R$ is a restriction of $M'$, and $M' \del E(R)$ has a spike-restriction $S'$ such that $\kappa_{M'}(E(R),E(S')) \ge 3$. By Theorem~\ref{linking}, we have $E(M') = E(R)\cup E(S')$. 
 Contracting any non-tip element of $S'$ that is not in $\cl_{M'}(E(R))$ gives a minor that contradicts the minimality of $M'$, so every non-tip element of $S'$ is spanned by $E(R)$. Since $S'$ has no coloops, it follows that $R$ is spanning in $M'$, giving the result. %Then, since $S'$ has no coloops, there is some non-tip element of 
%$S'$ that is not in $\cl_{M'}(E(R))$. Contracting this element gives a minor that contradicts the minimality of $M'$.  
\end{proof}

We use the above lemma to show that a matroid with a spike-restriction with sufficient connectivity to a large complete bipartite graph has a large nongraphic extension of a clique as a minor:

\begin{lemma}\label{bprestrictionwin}
Let $m \ge 3$ be an integer. If $M$ is a matroid with a spike-restriction $S$, and $M \del E(S)$ has an $M(K_{m+3,m+3})$-restriction $R$ so that $\kappa_M(E(R),E(S)) \ge 3$, then $M$ has a minor isomorphic to a nongraphic extension of $M(K_{m+1})$. 
\end{lemma}
\begin{proof}
By Lemma~\ref{linkspikerestr}, there is a minor $M_1$ of $M$ with $R$ as a spanning restriction and with a spike-restriction of rank at least $3$. Let $H \cong K_{m+3,m+3}$ be such that $R = M(H)$. Let $J$ be a matching of $H$ that is maximal so that $|J| \le m$ and $M_1 \con J$ has a spike-restriction $S$ of rank at least $3$. 

 If $|J| = m$, then $H \con J$ has a $K_{m+1}$-subgraph and is clearly $4$-connected. Therefore $M(H) \con J$ is a spanning vertically $4$-connected restriction of $M_1 \con J$ with an $M(K_{m+1})$-restriction $R'$. By vertical $4$-connectivity we have $\kappa_{M_1 \con J}(E(R'),E(S)) \ge 3$, so by Lemma~\ref{linkspikerestr} there is a minor $M_2$ of $M_1 \con J$ with $R'$ as a spanning restriction and with a nongraphic spike-restriction; this contains a nongraphic extension of $R'$, giving the lemma. 
 
 If $|J| < m$, then there are at least $8$ vertices of $H$ unsaturated by $J$, so there is a $6$-element independent set $I \subseteq E(H) - J$ such that $J \cup \{f\}$ is a matching for each $f \in I$. By maximality, we have $f \in \cl_{M_1 \con J}(E(S))$ for each $f \in I$, so $r(S) \ge 6$. Let $e \in I$ be not parallel to a tip of $S$ in $M_1 \con J$. By Lemma~\ref{breakspikeeasy}, there are spike-restrictions $S_1,S_2$ of $M_1 \con (J \cup \{e\})$ such that $E(S_1) \cup E(S_2) = E(S)-\{e\}$. But $E(S) - \{e\}$ has rank at least $5$ in $M_1 \con (J \cup \{e\})$, so $S_1$ or $S_2$ has rank at least $3$, contradicting the maximality of $J$. 
\end{proof}

\section{Tangles}

In this section we discuss tangles, structures that capture the idea of connectivity into a minor. Tangles were introduced for graphs, and implicitly for matroids, by Robertson and Seymour [\ref{gmx}] and were later extended explicitly to matroids [\ref{d95},\ref{ggrw06}]. The material in this section follows [\ref{ggwnonprime}] and [\ref{nvz14}].

Let $M$ be a matroid and let $\theta \in \bZ^+$. A set $X \subseteq E(M)$ is \emph{$k$-separating in $M$} if $\lambda_M(X) < k$. A collection $\cT$ of subsets of $E(M)$ is a \emph{tangle of order $\theta$} if 
\begin{enumerate}
\item Every set in $\cT$ is $(\theta-1)$-separating in $M$ and, for each $(\theta-1)$-separating set $X \subseteq E(M)$, either $X \in \cT$ or $E(M) - X \in \cT$; 
\item if $A,B,C \in \cT$ then $A \cup B \cup C \ne E(M)$; and
\item $E(M) - \{e\} \notin \cT$ for each $e \in E(M)$. 
\end{enumerate}
We refer to the sets in $\cT$ as \emph{$\cT$-small}. Given a tangle of order $\theta$ on a matroid $M$ and a set $X \subseteq E(M)$, we set $\kappa_{\cT}(X) = \theta - 1$ if $X$ is contained in no $\cT$-small set, and $\kappa_{\cT}(X) = \min\{\lambda_M(Z): X \subseteq Z \in \cT\}$ otherwise. The proof of our first lemma appears in [\ref{ggrw06}].

\begin{lemma}\label{tanglematroid}
If $\cT$ is a tangle of order $\theta$ on a matroid $M$, then $\kappa_{\cT}$ is the rank function of a rank-$(\theta-1)$ matroid on $E(M)$. 
\end{lemma}

This matroid, which we denote $M(\cT)$, is the \emph{tangle matroid}. We abbreviate closure function of this matroid by $\cl_{\cT}$. The next lemma is easily proved.

\begin{lemma}\label{tangleminor}
If $N$ is a minor of a matroid $M$ and $\cT_N$ is a tangle of order $\theta$ on $N$, then $\{X \subseteq E(M): \lambda_M(X) < \theta-1, X \cap E(N) \in \cT_N\}$ is a tangle of order $\theta$ on $M$.
\end{lemma}

This is the tangle on $M$ \emph{induced} by $\cT_N$. 

If $M$ is a matroid and $k$ is an integer, then we write $\cT_k(M)$ for the collection of $(k-1)$-separating sets of $M$ that are neither spanning nor cospanning. For example, if $M \cong M(K_{n+1})$ and $k = \twothirds{n}$, then $\cT_k(M)$ is simply the collection of subsets of $E(M)$ of rank at most $k-2$. Since $K_{n+1}$ is not the union of three subgraphs on at most $\tfrac{2}{3}n$ vertices, we easily have the following:

\begin{lemma}\label{pgtangle}
If $n \ge 2$ and $M \cong M(K_{n+1})$, then $\cT_{\twothirds{n}}(M)$ is a tangle of order $\twothirds{n}$ in $M$. 
\end{lemma}

If $M$ is a matroid with an $M(K_{n+1})$-minor $N$, then we write $\cT_{\twothirds{n}}(M,N)$ for the tangle of order $\left\lceil \tfrac{2n}{3} \right\rceil$ in $M$ induced by $\cT_{\twothirds{n}}(N)$. 

The next result is a slight variation of a lemma from [\ref{ggwnonprime}].

\begin{lemma}\label{tanglecontract}
Let $k \in \bZ^+$, let $M$ be a matroid and let $N$ be a minor of $M$ such that $\cT_k(N)$ is a tangle. If $X \subseteq E(M)$ is contained in a $\cT_k(M,N)$-small set, then there is a minor $M'$ of $M$ such that $M'|X = M|X$, $M'$ has $N$ as a minor, and $X$ is contained in a $\cT_k(M',N)$-small set $X'$ such that $E(M') = E(N) \cup X'$ and $\lambda_{M'}(X') = \kappa_{\cT_k(M',N)}(X) = \kappa_{\cT_k(M,N)}(X)$. 
\end{lemma}
\begin{proof}
Let $b = \kappa_{\cT_k(M,N)}(X)$ and let $M'$ be a minimal minor of $M$ such that $N$ is a minor of $M'$, so that $M|X = M'|X$ and so that $\kappa_{\cT_k(M',N)}(X) = b$. Let $\cT = \cT_k(M',N)$ and $X' = \cl_{\cT}(X)$. It remains to show that $E(M') = X' \cup E(N)$. If not, there is some $e \in E(M') - (X' \cup E(N))$. Since $\cl_{M'}(X) \subseteq X'$, we know that $M'|X$ is a restriction of both $M' \con e$ and $M' \del e$. If $N$ is a minor of $M' \con e$, then by the choice of $M'$ we have $\kappa_{\cT_k(M' \con e,N)}(X) \le b-1$. Therefore there is some set $Z \in \cT_k(M' \con e,N)$ such that $\lambda_{M' \con e}(Z) \le b-1$ and $X \subseteq Z$. Thus $Z \cup \{e\} \in \cT$ and $\lambda_{M'}(Z \cup \{e\}) \le b$ so $\kappa_{\cT}(X \cup \{e\}) = \kappa_{\cT}(X)$ and $e \in \cl_{\cT}(X)$, a contradiction. The case where $N$ is a minor of $M' \del e$ is similar. 
\end{proof}

The next lemma is our main technical application of tangles; it shows that a restriction $X$ of a matroid $M$ with a huge clique minor can be contracted onto a large clique restriction with as much connectivity as could be expected:

\begin{lemma}\label{tanglebig}
There is a function $f_{\ref{tanglebig}}: \bZ^2 \to \bZ$ so that, for all $m,n,\ell \in \bZ$ with $m > 0$, $\ell \ge 2$ and $n \ge f_{\ref{tanglebig}}(m,\ell)$, if $M \in \cU(\ell)$ has an $M(K_{n+1})$-minor $N$ with corresponding tangle $\cT = \cT_{\twothirds{n}}(M,N)$ and $X \subseteq E(M)$ satisfies $\kappa_{\cT}(X) \le m$, then $M$ has a minor $M'$ with an $M(K_{m+1})$-restriction $R$ so that $X \cap E(R) = \varnothing$, $M'|X = M|X$, $E(M') = E(R) \cup X$ and $\lambda_{M'}(X) = \kappa_{\cT}(X)$. 
\end{lemma}
\begin{proof}
Let $n_1 = f_{\ref{slogliftclique}}(\ell,m,m)$ and let $n = \max(2m,2n_1-1)$. 

Let $t = r_\cT(X)$ and $k = \twothirds{n}$. Note that $t \le m < k$. Since $r_{\cT}(X) = t$, the set $X$ is contained in a $\cT$-small set. By Lemma~\ref{tanglecontract}, there is a minor $M_1$ of $M$ such that $M_1|X = M|X$, $M_1$ has $N$ as a minor, and $X$ is contained in a $\cT_{k}(M_1,N)$-small set $X'$ such that $E(M_1) = E(N) \cup X'$ and $\lambda_{M_1}(X') = r_{\cT_k(M_1,N)}(X) = r_{\cT}(X) = t$. Since $N \cong M(K_{n+1})$ and $X' \cap E(N)$ is $\cT_k(N)$-small, it follows that $r(M_1 | (E(N)-X')) = r(M_1|E(N))$ and so we also have $\sqcap_{M_1}(X',E(N)) = t$. 

Let $C \subseteq E(M_1)$ be such that $N$ is a restriction of $M_1 \con C$. Let $N'$ be an $M(K_{n_1,n_1})$-restriction of $N$. Since $E(N') \subseteq E(N)$, we have $\sqcap_{M_1}(X',E(N')) \le \sqcap_{M_1}(X',E(N)) = t$. By Lemma~\ref{slogliftclique}, we see that $M_1|(E(N')-X')$ has an $M(K_{m,m})$-restriction $R'$. Note that $X \cap E(R') = \varnothing$ and $\kappa_{M_1}(X,E(R')) \le \lambda_{M_1}(X') \le t$. Moreover we have $r(R') = 2m-1 > t$, so, since $r_{\cT_k(M_1,E(N))}(X) = t$, we must have $\kappa_{M_1}(X,E(R')) = t$, as otherwise $M_1$ has a $t$-separation for which neither side is $\cT_k(M_1,N)$-small. 

By Theorem~\ref{linking}, the matroid $M_1$ has a minor $M_2$ such that $E(M_2) = X \cup E(R')$, $M_2|X = M_1|X, M_2|E(R') = R'$, and $\lambda_{M_2}(X) = t$. Let $R = M(H)$, where $H \cong K_{m(m+1),m(m+1)}$, and let $H_1, \dotsc, H_{m+1}$ be vertex-disjoint $K_{m,m}$-subgraphs of $H$. Now the sets $E(H_i)$ are mutually skew in $M_2$, so $\sum_{i = 1}^{m+1} \sqcap_{M_2}(X,E(H_i)) \le \sqcap_{M_2}(X,E(H)) = t \le m$, so there is some $i$ such that $\sqcap_{M_2}(X,E(H_i)) = 0$. Let $J$ be the edge set of an $(m-1)$-edge matching of $H_i$ and let $M_3 = M_2 \con J$. Now $M_3|(H_i-J)$ has a $K_{m+1}$-restriction $R$, and $\lambda_{M_3}(X) = \lambda_{M_2}(X) = t$.

 Let $B$ be a basis for $M_3$ containing a basis $B'$ for $M_3 \del X$. Note that $M_3 \con (B-B')$ has $M(H \con J)$ as a spanning restriction and $H \con J$ is an $(m+1)$-connected graph, so $M_3 \con (B-B')$ is vertically $(m+1)$-connected. Since $B-B'$ is skew to $E(M_3 \del X)$, we have 
 \begin{align*}
 	\kappa_{M_3}(X,E(R)) &= \kappa_{M_3 \con (B-B')}(X - (B-B'),E(R)) \\ 
				 &\ge \min(m,r_{M_3 \con (B-B')}(X - (B-B')),r_{M_3 \con (B-B')}(E(R)))\\
				 &= \min(t,m,m) = t.
\end{align*}
Theorem~\ref{linking} now gives the required minor. 
\end{proof}

When $M$ is vertically $(t+1)$-connected and $r_M(X) \le t$ in the above lemma, we have $\kappa_{\cT}(X) = r_M(X)$, and we obtain a simpler corollary:

\begin{corollary}\label{easiercontract}
There is a function $f_{\ref{easiercontract}}: \bZ^2 \to \bZ$ so that, for all $t,m,n,\ell \in \bZ$ with $m \ge t > 0$, $\ell \ge 2$ and $n \ge f_{\ref{easiercontract}}(m,\ell)$, if $M \in \cU(\ell)$ is a vertically $(t+1)$-connected matroid with an $M(K_{n+1})$-minor and $X \subseteq E(M)$ satisfies $r_M(X) \le t$, then $M$ has a rank-$m$ minor $N$ with an $M(K_{m+1})$-restriction such that $X \subseteq E(N)$ and $N|X = M|X$. 
\end{corollary}

\section{The main result}

We can now prove our main theorem. First we show that a spike with connectivity $3$ to a huge clique minor gives a nongraphic extension of a large clique in a minor:

\begin{lemma}\label{spikecliquewin}
There is a function $f_{\ref{spikecliquewin}}:\bZ^2 \to \bZ$ so that, for each $m,\ell,n \in \bZ$ with $m \ge 3$, $\ell \ge 2$,
and   $n \ge f_{\ref{spikecliquewin}}(m,\ell)$,
if $M \in \cU(\ell)$ is a matroid with an $M(K_{n+1})$-minor $N$ and a spike-restriction whose ground set has connectivity at least $3$ to the tangle $\cT_{\twothirds{n}}(M,N)$, then $M$ has a minor isomorphic to a nongraphic extension of $M(K_{m+1})$. 
\end{lemma}
\begin{proof}
Let $m \ge 3$ and $\ell \ge 2$ be integers. Let $n' = f_{\ref{singleliftclique}}(\ell,m+3)$. Set $f_{\ref{spikecliquewin}}(m,\ell) = \max(2n',f_{\ref{tanglebig}}(\ell,m))$. 

Let $n \ge f_{\ref{spikecliquewin}}(m,\ell)$ and let $k = \twothirds{n}$. Let $M \in \cU(\ell)$ be a matroid with an $M(K_{n+1})$-minor $N$ and a spike-restriction $S_0$ such that $\kappa_{\cT_{k}(M,N)}(E(S_0)) \ge 3$. We show that $M$ has a nongraphic extension of $M(K_{m+1})$ as a minor; by considering a parallel extension of $M$ if necessary, we may assume that $E(S_0) \cap E(N) = \varnothing$. Let $M_1$ be a minimal minor of $M$ such that
\begin{enumerate}[(1)]
\item $N$ is a minor of $M_1$, and  
\item $M_1 \del E(N)$ has a spike-restriction $S$ such that $\kappa_{\cT_k(M_1,N)}(E(S)) \ge 3$.
\end{enumerate}
Let $C$ be an independent set in $M_1$ such that $N$ is a spanning restriction of $M_1 \con C$. If $|C| \le 1$ then $N = (M_1 \con C)|E(N)$ has an $M(K_{n',n'})$-restriction, so by Lemma~\ref{singleliftclique} the matroid $M_1|E(N)$ has an $M(K_{m+3,m+3})$-restriction $R_1$. Moreover, we clearly have $\kappa_{\cT_k(M_1,N)}(E(R_1)) \ge 2(m+3)-1 \ge 3$, so $\kappa_{M_1}(E(S),E(R_1)) \ge 3$, as otherwise we have a $(\le 3)$-separation with both sides $\cT_k(M_1,N)$-small. By Lemma~\ref{bprestrictionwin}, the result holds. 

If $|C| \ge 2$ then there is some $e \in C$ that is not parallel in $M$ to a tip of $S$. By Lemma~\ref{breakspikeeasy}, there are spike-restrictions $S_1, S_2$ of $M_1 \con e$ such that $E(S_1) \cup E(S_2) = E(S)$. By minimality of $M_1$, we have $\kappa_{\cT_k(M_1 \con e,N)}(E(S_i)) \le 2$ for each $i \in \{1,2\}$. It follows since $\kappa_{\cT_k(M_1 \con e,N)}$ is the rank function of a matroid on $M_1 \con e$ that $\kappa_{\cT_k(M_1 \con e,N)}(E(S)) \le 2 + 2 = 4$ and so $\kappa_{\cT_k(M_1,N)}(E(S)) \le 5$. 

By Lemma~\ref{tanglebig} and the definition of $n$, there is a minor $M_2$ of $M_1$ with an $M(K_{m+1})$-restriction $R_2$ such that $E(R_2) \cap E(S) = \varnothing$, $E(M_2) = E(R_2) \cup E(S)$, $3 \le \lambda_{M_2}(E(S)) \le 5$ and $S = M_2|E(S)$. Since $\kappa_{M_2}(E(S),E(R_2)) = \lambda_{M_2}(E(S)) \ge 3 $, Lemma~\ref{linkspikerestr} implies that $M_2$ has a minor with $R_2$ as a spanning restriction and with a nongraphic spike-restriction. The result follows.  
\end{proof}

Finally, we restate and prove Theorem~\ref{main}. 

\begin{theorem}\label{maintech}
Let $m \ge 3$ and $\ell \ge 2$ be integers. If $\cM$ is the class of matroids with no $U_{2,\ell+2}$-minor and with no nongraphic single-element extension of $M(K_{m+1})$ as a minor, then $h_{\cM}(n) \approx \binom{n+1}{2}$. 
\end{theorem}
\begin{proof}
Suppose that the theorem fails. Clearly $\cM$ contains the graphic matroids, so $h_{\cM}(n) \ge \binom{n+1}{2}$ for all $n$; thus, we have $h_{\cM}(n) > \binom{n+1}{2}$ for infinitely many $n$. 

Let $n_0 = \max(f_{\ref{easiercontract}}(m,\ell),f_{\ref{spikecliquewin}}(m,\ell))$ and $n_1 = \max(m,2\alpha_{\ref{lgrt}}(n_0,\ell)).$ 
By Theorem~\ref{maintech2} with $p(x) = \binom{x+1}{2}$, $s = 4$ and $r = n_1$, we see that there exists $M \in \cM$ such that $r(M) \ge n_1$, $\elem(M) > \binom{r(M)+1}{2}$ and either 
\begin{enumerate}
\item\label{mti} $M$ has a spanning clique, or 
\item\label{mtii} $M$ is vertically $4$-connected and there is some nonloop $e$ of $M$ such that $\elem(M) - \elem(M \con e) > r(M)$.
\end{enumerate}
We may assume that $M$ is simple. If (\ref{mti}) holds, then since $|M| > \binom{r(M)+1}{2}$, the matroid $M$ has a nongraphic extension of a rank-$r(M)$ clique as a restriction. Since $r(M') \ge n_1 \ge m \ge 3$, it is easy to repeatedly contract elements of $M'$ and simplify to obtain a nongraphic extension of $M(K_{m+1})$, a contradiction. Therefore (\ref{mtii}) holds. 		

 Now $r(M) \ge 2\alpha_{\ref{lgrt}}(n_0,\ell)$, so $\elem(M) > \binom{r(M)+1}{2} > \alpha_{\ref{lgrt}}(n_0,\ell) r(M)$; thus, $M$ has an $M(K_{n_0+1})$-minor $N$ by Theorem~\ref{lgrt}. 

Let $\cL$ be the set of lines of $M$ containing $e$. If $|L| \ge 4$ for some $L \in \cL$, then by vertical $3$-connectivity of $M$, Corollary~\ref{easiercontract} implies that $M$ has a rank-$m$ minor $M'$ with an $M(K_{m+1})$-restriction such that $M'|L = M|L$. Since $M'|L$ is nongraphic, this minor contains a nongraphic extension of $M(K_{m+1})$, a contradiction. So $|L| \le 3$ for each $L \in \cL$, and each parallel class of $M \con e$ has size $1$ or $2$. 

Let $\cL_3 = \{L \in \cL: |L| = 3\}$. Note that $r(M) < \elem(M) - \elem(M \con e) = 1 + |\cL_3|$, so $r(M) \le |\cL_3|$. Therefore there are at least $r(M) > r(M \con e)$ parallel pairs in $M \con e$, so there is a circuit $C$ of $M \con e$ such that $|C| \ge 3$ and each $x \in C$ lies in a parallel class of size $2$ in $M \con e$. Therefore $e$ is the tip of a nongraphic spike-restriction $S$ of $M$. Since $M$ is vertically $4$-connected, the set $E(S)$ has rank at least $3$ in the tangle $\cT_{\twothirds{n_0}}(M,N)$. By the definition of $n_0$, Lemma~\ref{spikecliquewin} gives a nongraphic extension of $M(K_{m+1})$ as a minor of $M$, again a contradiction. 		
\end{proof}

\section*{References}

\newcounter{refs}

\begin{list}{[\arabic{refs}]}
{\usecounter{refs}\setlength{\leftmargin}{10mm}\setlength{\itemsep}{0mm}}

\item\label{archer}
S. Archer,
Near varieties and extremal matroids, 
Ph.D. thesis, Victoria University of Wellington, 2005. 

\item\label{d95}
J.S. Dharmatilake, 
A min-max theorem using matroid separations, 
Matroid Theory Seattle, WA, 1995, 
Contemp. Math. vol. 197, Amer. Math. Soc., Providence RI (1996), pp. 333--342.

\item \label{esthm}
P. Erd\H os, A.H. Stone,
On the structure of linear graphs,
Bull. Amer. Math. Soc. 52 (1946) 1087-1091.

\item \label{ggrw06}
J. Geelen, B. Gerards, N. Robertson, G. Whittle, 
Obstructions to branch decomposition of matroids, 
J. Combin. Theory. Ser. B 96 (2006) 560--570.

\item\label{ggw06}
J. Geelen, B. Gerards, G. Whittle, 
Excluding a planar graph from $\GF(q)$-representable matroids,
J. Combin. Theory Ser. B 97 (2007) 971--998.

\item \label{gkw09}
J. Geelen, J.P.S. Kung, G. Whittle, 
Growth rates of minor-closed classes of matroids, 
J. Combin. Theory. Ser. B 99 (2009) 420--427.

\item\label{ggwnonprime}
J. Geelen, B. Gerards, G. Whittle, 
Matroid structure. I. Confined to a subfield, 
in preparation. 

\item\label{heller}
I. Heller, 
On linear systems with integral valued solutions,
Pacific. J. Math. 7 (1957) 1351--1364.

\item\label{kmpr}
J.P.S. Kung, D. Mayhew, I. Pivotto, G.F. Royle,
Maximum size binary matroids with no $\AG(3,2)$-minor are graphic,
SIAM J. Discrete Math. 28 (2014), 1559--1577.

\item\label{kung91}
J.P.S. Kung,
Extremal matroid theory, in: Graph Structure Theory (Seattle WA, 1991), 
Contemporary Mathematics, 147, American Mathematical Society, Providence RI, 1993, pp.~21--61.

\item\label{mcguinness}
S. McGuinness, 
Binary matroids with no 4-spike minors, 
Discrete Math. 324 (2014) 72--77.

\item\label{nthesis}
P. Nelson, 
Exponentially dense matroids. 
Ph.D. Thesis, University of Waterloo, 2011. 

\item\label{nvz14}
P. Nelson, S.H.M. van Zwam, 
Matroids representable over fields with a common subfield, 
arXiv:1401.7040 [math.CO].

\item \label{oxley}
J. G. Oxley, 
Matroid Theory (2nd edition),
Oxford University Press, New York, 2011.

\item\label{gmx}
N. Robertson, P. D. Seymour, 
Graph Minors. X. Obstructions to tree-decomposition, 
J. Combin. Theory. Ser. B 52 (1991) 153--190. 

\item\label{zaslav}
T. Zaslavsky, Signed graphs, 
Discrete Appl. Math. 4 (1982) 47--74

\end{list}

\end{document}